\documentclass{amsart}

\usepackage[all]{xy}

\newtheorem{theorem}{Theorem}[section]
\newtheorem{corollary}[theorem]{Corollary}
\newtheorem{lemma}[theorem]{Lemma}
\newtheorem{proposition}[theorem]{Proposition}
\theoremstyle{definition}
\newtheorem{definition}[theorem]{Definition}
\newtheorem{example}[theorem]{Example}
\newtheorem{remark}[theorem]{Remark}
\numberwithin{equation}{section}

\newcommand{\add}{\operatorname{add}}
\newcommand{\Cok}{\operatorname{Cok}}
\newcommand{\End}{\operatorname{End}}
\newcommand{\Ext}{\operatorname{Ext}}
\newcommand{\Hom}{\operatorname{Hom}}
\newcommand{\Id}{\operatorname{Id}}
\renewcommand{\Im}{\operatorname{Im}}
\newcommand{\Ker}{\operatorname{Ker}}
\newcommand{\proj}{\operatorname{proj}}
\newcommand{\rad}{\operatorname{rad}}

\newcommand{\CA}{\mathcal{A}}
\newcommand{\CC}{\mathcal{C}}
\newcommand{\CE}{\mathcal{E}}
\newcommand{\CI}{\mathcal{I}}
\newcommand{\CP}{\mathcal{P}}
\newcommand{\Endc}{\End_\mathcal{C}}
\newcommand{\Extc}{\Ext_\mathcal{C}^1}
\newcommand{\Homc}{\Hom_\mathcal{C}}
\newcommand{\Homco}{\overline\Hom_\mathcal{C}}
\newcommand{\Homcu}{\underline\Hom_\mathcal{C}}
\newcommand{\radc}{\rad_\mathcal{C}}
\newcommand{\op}{\mathrm{op}}
\newcommand{\set}[1]{\left\{#1\right\}}
\newcommand{\smatrix}[1]{\left(\begin{smallmatrix}#1\end{smallmatrix}\right)}

\begin{document}

\title{The Auslander-Reiten duality via morphisms determined by objects}

\author{Pengjie Jiao}
\address{School of Mathematical Sciences, University of Science and Technology of China, Hefei 230026, PR China}
\email{jiaopjie@mail.ustc.edu.cn}

\author[Jue Le]{Jue Le$^*$}
\thanks{$^*$The corresponding author.}
\address{School of Mathematical Sciences, University of Science and Technology of China, Hefei 230026, PR China}
\email{juele@ustc.edu.cn}

\date{\today}

\begin{abstract}
  Given an exact category $\mathcal{C}$, we denote by $\mathcal{C}_l$ the smallest additive subcategory containing injectives and indecomposable objects which appear as the first term of an almost split conflation.
  We prove that a deflation is right determined by some object if and only if its intrinsic kernel lies in $\mathcal{C}_l$.
  We give characterizations for $\mathcal{C}$ having Auslander-Reiten duality.
\end{abstract}

\maketitle

\section{Introduction}

Let $k$ be a commutative artinian ring and $\check k$ be the minimal injective cogenerator. We denote by $D=\Hom_k(-,\check k)$ the Matlis duality. The categories we consider are $k$-linear Hom-finite Krull-Schmidt and skeletally small.

Recall from \cite{LenzingZuazua2004Auslander} that an abelian category $\CA$ is said to \emph{have Auslander-Reiten duality}, if for any objects $X$ and $Y$, there exist natural isomorphisms
\[
  \overline\Hom_\CA(Y,\tau X)
  \simeq D\Ext_\CA^1(X,Y)
  \simeq \underline\Hom_\CA(\tau^-Y,X).
\]
Here $\tau$ and $\tau^-$ are mutually quasi-inverse equivalences between the stable categories of $\CA$. More generally, the notion of \emph{generalized Auslander-Reiten duality} for an exact category $\CC$ was introduced in \cite[Section~3]{JiaoGeneralized}. We denote by $\CC_r$ the smallest additive subcategory of $\CC$ containing projectives and indecomposable objects which appear as the third term of an almost split conflation, and by $\CC_l$ the smallest additive subcategory containing injectives and indecomposable objects which appear as the first term of an almost split conflation. We still have the mutually quasi-inverse equivalences $\tau$ and $\tau^-$ between stable categories of $\CC_r$ and $\CC_l$.

Auslander introduced the notion of \emph{morphisms determined by objects} in his Philadelphia note \cite{Auslander1978Functors}, which extends Auslander-Reiten theory in some aspects. We study morphisms determined by objects in an exact category $\CC$. We prove that a deflation in $\CC$ is right determined by some object if and only if its intrinsic kernel lies in $\mathcal{C}_l$; see Theorem~\ref{thm:det}. We give characterizations for an object lying in $\CC_r$ via morphisms determined by objects; see Theorem~\ref{thm:C}.

Following \cite{Krause2013Morphisms} and \cite{ChenLe2015note}, we introduce the notions of \emph{having right stably determined deflations} and \emph{having left stably determined inflations}. We show that the conditions ``$\CC$ has Auslander-Reiten duality'', ``$\CC$ has right stably determined deflations'' and ``$\CC$ has left stably determined inflations'' are equivalent; see Theorem~\ref{thm:AR}.

The paper is organized as follows. In Sections~2 we recall some basic properties of morphisms determined by objects. In Section~3 we prove that if  two objects $C$ and $C'$ are isomorphic in the projectively stable category, then a deflation is right determined by $C$ if and only if it is right determined by $C'$. In Section~4 we prove the following existence theorem: given objects $C\in\CC_r$ and $Y\in\CC$, for certain submodules $H$ of $\Homc(C,Y)$, there exists a deflation $\alpha\colon X\to Y$ with $H=\Im\Homc(C,\alpha)$. Sections~5, 6 and 7 are dedicated to the proofs of Theorems~\ref{thm:det}, \ref{thm:C} and \ref{thm:AR}, respectively.

\section{Morphisms determined by objects}

Let $\CC$ be an additive category. We recall some well-known notions. Here, we use the terminologies in \cite[Section~3]{Ringel2013Auslander}.
Let $f\colon X\to Y$ be a morphism and $C$ be an object. We call $f$ \emph{right $C$-determined} (or right determined by $C$) and call $C$ a \emph{right determiner} of $f$, if the following condition is satisfied: each morphism $g\colon T\to Y$ factors through $f$, provided that for each $h\colon C\to T$ the morphism $g\circ h$ factors through $f$. If moreover $C$ is a direct summand of any right determiner of $f$, we call $C$ a \emph{minimal right determiner} of $f$.

\begin{lemma}\label{lem:det.sum}
  Let $f_1\colon X_1\to Y_1$ and $f_2\colon X_2\to Y_2$ be two morphisms. Then for an object $C$, the morphism $\smatrix{f_1 &0\\ 0 &f_2}\colon X_1\oplus X_2\to Y_1\oplus Y_2$ is right $C$-determined if and only if both $f_1$ and $f_2$ are right $C$-determined.
\end{lemma}

\begin{proof}
  For the sufficiency, we assume that $f_1$ and $f_2$ are right $C$-determined. Let $\smatrix{g_1\\ g_2}\colon T\to Y_1\oplus Y_2$ be a morphism such that for each $h\colon C\to T$, there exists some morphism $\smatrix{s_1\\ s_2}\colon C\to X_1\oplus X_2$ satisfying $\smatrix{g_1\\ g_2} h=\smatrix{f_1 &0\\ 0 &f_2} \smatrix{s_1\\ s_2}$. We obtain $g_1\circ h=f_1\circ s_1$ and $g_2\circ h=f_2\circ s_2$. Since $f_1$ and $f_2$ are right $C$-determined, there exist some morphisms $t_1\colon T\to X_1$ and $t_2\colon T\to X_2$ such that $g_1=f_1\circ t_1$ and $g_2=f_2\circ t_2$. The morphism $\smatrix{t_1\\t_2}\colon T\to X_1\oplus X_2$ satisfies $\smatrix{g_1\\ g_2} = \smatrix{f_1 &0\\ 0 &f_2} \smatrix{t_1\\ t_2}$. It follows that $\smatrix{f_1 &0\\ 0 &f_2}$ is right $C$-determined.

  For the necessity, we assume that $\smatrix{f_1 &0\\ 0 &f_2}$ is right $C$-determined. Let $g\colon T\to Y_1$ be a morphism such that for each $h\colon C\to T$, there exists some morphism $s\colon C\to X_1$ satisfying $g\circ h=f_1\circ s$. The morphisms $\smatrix{g\\0}\colon T\to Y_1\oplus Y_2$ and $\smatrix{s\\0}\colon C\to X_1\oplus X_2$ satisfy $\smatrix{g\\ 0} h=\smatrix{f_1 &0\\ 0 &f_2} \smatrix{s\\ 0}$. Since $\smatrix{f_1 &0\\ 0 &f_2}$ is right $C$-determined, there exists some morphism $\smatrix{t_1\\ t_2}\colon T\to X_1\oplus X_2$ such that $\smatrix{g\\ 0} = \smatrix{f_1 &0\\ 0 &f_2} \smatrix{t_1\\ t_2}$. We obtain $g=f_1\circ t_1$. It follows that $f_1$ is right $C$-determined. Similarly, the morphism $f_2$ is right $C$-determined.
\end{proof}

\begin{lemma}\label{lem:det.PB}
  Let $f\colon X\to Y$ and $g\colon Z\to Y$ be two morphisms. Assume that $f'\colon E\to Z$ and $g'\colon E\to X$ form the pullback of $f$ and $g$. If $f$ is right $C$-determined for some object $C$, then $f'$ is also right $C$-determined.
\end{lemma}

\begin{proof}
  Let $h\colon T\to Z$ be a morphism such that for each $u\colon C\to T$, there exists some morphism $v\colon C\to E$ satisfying $h\circ u=f'\circ v$. Then we have $g\circ h\circ u=g\circ f'\circ v=f\circ g'\circ v$.
  Since $f$ is right $C$-determined, there exists some morphism $s\colon T\to X$ such that $g\circ h=f\circ s$. Since $f'$ and $g'$ form the pullback of $f$ and $g$, there exists a morphism $t\colon T\to E$ such that $h=f'\circ t$. It follows that $f'$ is right $C$-determined.
\end{proof}

Let $f\colon X\to Y$ be a morphism. Recall that $f$ is called \emph{right almost split} if $f$ is not a retraction and each morphism $g\colon Z\to Y$ which is not a retraction factors through $f$. Dually, we call $f$ \emph{left almost split} if $f$ is not a section and each morphism $g\colon X\to Z$ which is not a section factors through $f$.

Given two objects $X$ and $Y$, we denote by $\radc(X,Y)$ the set of morphisms $f\colon X\to Y$, that for any object $Z$ and any morphisms $g\colon Z\to X$ and $h\colon Y\to Z$, the morphism $h\circ f\circ g$ lies in $\rad\Endc(Z)$. Then $\radc$ forms an ideal of $\CC$. We observe by \cite[Corollary~2.10]{Krause2015Krull} that
\begin{equation}\label{eq:rad}
  \radc(X,Y) = \set{ f \colon X \to Y \middle|
    f\circ g\in\rad\Endc(Y),
    \mbox{ for each }
    g\colon Y\to X}.
\end{equation}

A morphism $g\colon Z\to Y$ is said to \emph{almost factor through} $f\colon X\to Y$, if $g$ does not factor through $f$, and for each object $T$ and each morphism $h\in\radc(T,Z)$, the morphism $g \circ h$ factors through $f$.

\begin{proposition}\label{prop:ras}
  Let $f\colon X\to Y$ and $g\colon Z\to Y$ be two morphisms, such that $\Endc(Z)$ is a local ring and $g$ almost factors through $f$. Assume that $f'\colon E\to Z$ is the pullback of $f$ along $g$. Then $f'$ is right almost split.
\end{proposition}

\begin{proof}
  We observe that $f'$ is not a retraction, since $g$ does not factors through $f$. Given an object $T$, assume that $h\colon T\to Z$ is not a retraction. We observe by (\ref{eq:rad}) that $h\in\radc(T,Z)$, since $\Endc(Z)$ is local. Since $g$ almost factors through $f$, there exists some morphism $s\colon T\to X$ such that $g\circ h = f\circ s$. By the pullback diagram, there exists a morphism $t\colon T\to E$ such that $h=f'\circ t$. It follows that $f'$ is right almost split.
\end{proof}

\section{Deflations determined by objects}

Let $k$ be a commutative artinian ring. We denote by $k$-mod the category of finitely generated $k$-modules. Let $\check{k}$ be the minimal injective cogenerator for $k$. Then the Matlis duality $D=\Hom_k(-,\check{k})$ is a self-duality of $k$-mod. From now on, the categories we consider are $k$-linear Hom-finite and Krull-Schmidt.

Let $\CC$ be an exact category. Recall that an \emph{exact category} is an additive category $\CC$ together with a collection $\CE$ of kernel-cokernel pairs which satisfies the axioms in \cite[Appendix~A]{Keller1990Chain}. Here, a kernel-cokernel pair means a sequence of morphisms $X\xrightarrow{i}E\xrightarrow{d}Y$, which we denote by $(i,d)$, such that $i$ is the kernel of $d$, and $d$ is the cokernel of $i$. Each kernel-cokernel pair $(i,d)\in\CE$ is called a \emph{conflation}, while $i$ is called an \emph{inflation} and $d$ is called a \emph{deflation}.
Given a conflation $\eta\colon X\to E\to Y$, for each $f\colon Z\to Y$ we denote by $\eta.f$ the conflation $\Extc(f,X)(\eta)$; for each $g\colon X\to Z$ we denote by $g.\eta$ the conflation $\Extc(Y,g)(\eta)$.

Let $\alpha_1\colon X_1\to Y_1$ and $\alpha_2\colon X_2\to Y_2$ be two morphisms. We mention the fact that $\smatrix{\alpha_1 &0\\ 0 &\alpha_2}\colon X_1\oplus X_2\to Y_1\oplus Y_2$ is a deflation if and only if both $\alpha_1$ and $\alpha_2$ are deflations;
see \cite[Proposition~2.9 and Corollary~2.18]{Buhler2010Exact}.

Recall from \cite[Section~2]{LenzingZuazua2004Auslander} that a morphism $f\colon X\to Y$ is called \emph{projectively trivial} if for each $Z$, the induced map $\Extc(f,Z)\colon \Extc(Y,Z)\to \Extc(X,Z)$ is zero. We observe that $f$ is projectively trivial if and only if $f$ factors through each deflation ending at $Y$. Dually, we call $f$ \emph{injectively trivial} if for each $Z$, the induced map $\Extc(Z,f)\colon \Extc(Z,X)\to \Extc(Z,Y)$ is zero.

Given two objects $X$ and $Y$, we denote by $\CP(X,Y)$ the set of projectively trivial morphisms from $X$ to $Y$. Then $\CP$ forms an ideal of $\CC$. The \emph{projectively stable category} $\underline\CC$ of $\CC$ is the factor category $\CC/\CP$. Given a morphism $f\colon X\to Y$, we denote by $\underline f$ its image in $\underline\CC$. We denote by $\Homcu(X,Y) = \Homc(X,Y)/\CP(X,Y)$ the set of morphisms from $X$ to $Y$ in $\underline\CC$.

Dually, we denote by $\CI(X,Y)$ the set of injectively trivial morphisms from $X$ to $Y$. The \emph{injectively stable category} $\overline\CC$ of $\CC$ is the factor category $\CC/\CI$. Given a morphism $f\colon X\to Y$, we denote by $\overline f$ its image in $\overline\CC$. We denote by $\Homco(X,Y)=\Homc(X,Y)/\CI(X,Y)$ the set of morphisms from $X$ to $Y$ in $\overline\CC$.

\begin{lemma}\label{lem:ft.stable}
  Let $\alpha\colon X \to Y$ be a deflation. Then a morphism $f\colon T\to Y$ factors through $\alpha$ in $\CC$ if and only if $\underline f$ factors through $\underline\alpha$ in $\underline\CC$.
\end{lemma}

\begin{proof}
  It is sufficient to show the sufficiency. Assume that $\underline f$ factors through $\underline\alpha$ in $\underline\CC$. Then there exists some morphism $g\colon T\to X$ in $\CC$ such that $\underline f = \underline\alpha\circ\underline g$ in $\underline\CC$. We have that $f-\alpha\circ g$ is projectively trivial in $\CC$. Since $\alpha$ is a deflation, there exists some morphism $h\colon T\to X$ such that $f-\alpha\circ g=\alpha\circ h$ in $\CC$. It follows that $f=\alpha\circ(g+h)$, factoring through $\alpha$ in $\CC$.
\end{proof}

\begin{lemma}\label{lem:det.stable}
  Let $C$ be an object. Then a deflation $\alpha\colon X \to Y$ is right $C$-determined in $\CC$ if and only if $\underline\alpha$ is right $C$-determined in $\underline\CC$.
\end{lemma}

\begin{proof}
  For the sufficiency, we assume that $\underline\alpha$ is right $C$-determined in $\underline\CC$. Let $f\colon T\to Y$ be a morphism in $\CC$ such that for each $g\colon C\to T$, the morphism $f\circ g$ factors through $\alpha$ in $\CC$. Then $\underline f \circ \underline g$ factors through $\underline\alpha$ in $\underline\CC$. Since $\underline\alpha$ is right $C$-determined in $\underline\CC$, we have that $\underline f$ factors through $\underline\alpha$ in $\underline\CC$. By Lemma~\ref{lem:ft.stable}, the morphism $f$ factors through $\alpha$ in $\CC$. It follows that $\alpha$ is right $C$-determined in $\CC$.

  For the necessity, we assume that $\alpha$ is right $C$-determined in $\CC$. Let $f\colon T\to Y$ be a morphism in $\CC$ such that for each $g\colon C\to T$, the morphism $\underline{f} \circ \underline{g}$ factors through $\underline\alpha$ in $\underline\CC$. By Lemma~\ref{lem:ft.stable}, the morphism $f\circ g$ factors through $\alpha$ in $\CC$. Since $\alpha$ is right $C$-determined in $\CC$, we have that $f$ factors through $\alpha$ in $\CC$. Then $\underline f$ factors through $\underline\alpha$ in $\underline\CC$. It follows that $\underline\alpha$ is right $C$-determined in $\underline\CC$.
\end{proof}

\begin{proposition}\label{prop:det}
  Let $C$ and $C'$ be two objects. Assume that $C\simeq C'$ in $\underline\CC$. Then a deflation $\alpha\colon X \to Y$ is right $C$-determined if and only if it is right $C'$-determined.
\end{proposition}

\begin{proof}
  We observe that $\underline\alpha$ is right $C$-determined in $\underline\CC$ if and only if $\underline\alpha$ is right $C'$-determined in $\underline\CC$. Then the result follows by applying Lemma~\ref{lem:det.stable} twice.
\end{proof}

\section{An existence theorem}

Let $\CC$ be a Hom-finite Krull-Schmidt exact category. Given an object $C$, we denote by $\Gamma_C=\Endc(C)^{\op}$ the opposite algebra of the endomorphism algebra of $C$, and by $\add C$ the category of direct summands of finite direct sums of $C$.

Recall from \cite[Section~2]{JiaoGeneralized} two full subcategories $\CC_r$ and $\CC_l$ of $\CC$ as follows:
\[
  \CC_r = \set{X\in\CC \middle| \mbox{the functor }
  D\Extc(X,-) \colon \overline\CC \to k
  \mbox{-mod is representable}}
\]
and
\[
  \CC_l = \set{X\in\CC \middle| \mbox{the functor }
  D\Extc(-,X) \colon \underline\CC \to k
  \mbox{-mod is representable}}.
\]
Then we have the mutually quasi-inverse equivalences
\[
  \tau \colon \underline{\CC_r}
  \overset\sim\longrightarrow
  \overline{\CC_l}
  \mbox{ and }
  \tau^- \colon \overline{\CC_l}
  \overset\sim\longrightarrow
  \underline{\CC_r}.
\]
For each $X\in\CC_r$, we have a natural isomorphism
\[\Homco(-,\tau X)  \overset\sim\longrightarrow D\Extc(X,-);\]
for each $X\in\CC_l$, we have a natural isomorphism
\[\Homcu(\tau^-X,-) \overset\sim\longrightarrow D\Extc(-,X).\]

Recall that a conflation $\eta\colon X \xrightarrow{u} E \xrightarrow{v} Y$ is \emph{almost split} if $u$ is left almost split and $v$ is right almost split. Assume that $\eta$ is almost split. We mention that both $X$ and $Y$ are indecomposable. By the following lemma, which is due to \cite[Proposition~2.4]{JiaoGeneralized}, we have $X\in\CC_l$ and $Y\in\CC_r$. We mention that $Y\simeq\tau^-X$ in $\underline\CC$ and $X\simeq\tau Y$ in $\overline\CC$; see the proof of \cite[Lemma~3.2]{JiaoGeneralized}.

\begin{lemma}\label{lem:ass}
    Let $X$ be an indecomposable object.
  \begin{enumerate}
    \item\label{item:lem:X:1}
      Assume that $X$ is non-projective. Then $X\in\CC_r$ if and only if there exists an almost split conflation ending at $X$.
    \item\label{item:lem:X:2}
      Assume that $X$ is non-injective. Then $X\in\CC_l$ if and only if there exists an almost split conflation starting at $X$.
    \qed
  \end{enumerate}
\end{lemma}

The classical case of the following lemma is well known; see \cite[Corollary~3.5]{Ringel2013Auslander}.

\begin{lemma}\label{lem:ker.det}
  Let $\alpha\colon X\to Y$ be a deflation with $\Ker\alpha\in\CC_l$. Then $\alpha$ is right $\tau^-(\Ker\alpha)$-determined.
\end{lemma}

\begin{proof}
  Let $f\colon T\to Y$ be a morphism such that for each $g\colon\tau^-(\Ker\alpha)\to T$, the morphism $f \circ g$ factors through $\alpha$. We denote by $\eta$ the conflation $\Ker\alpha\to X\xrightarrow\alpha Y$. We obtain that $\eta.(f \circ g)$ splits.
  Since $\Ker\alpha\in\CC_l$, there exists a natural isomorphism
  \[
    \phi\colon \Homcu(\tau^-(\Ker\alpha),-)
    \longrightarrow D\Extc(-,\Ker\alpha).
  \]
  Set
  $\gamma = \phi_{\tau^-(\Ker\alpha)}(\underline{\Id_{\tau^-(\Ker\alpha)}})$.
  By the naturality of $\phi$, we have
  \[
    \phi_T(\underline g)
    = D\Extc(g,\Ker\alpha)(\gamma)
    = \gamma\circ\Extc(g,\Ker\alpha).
  \]
  It follows that
  \[
    \phi_T(\underline g)(\eta.f)
    = \gamma((\eta.f).g)
    = \gamma(\eta.(f\circ g))
    = 0.
  \]
  We observe that $\phi_T(\underline g)$ runs over all maps in $D\Extc(T,\Ker\alpha)$, when $\underline g$ runs over all morphisms in $\Homcu(\tau^-(\Ker\alpha),T)$.
  It follows that $\eta.f$ splits. In other words, the morphism $f$ factors through $\alpha$. Then the result follows.
\end{proof}

Given an object $C$ and a morphism $f\colon X\to Y$, we denote by $\Im\Homc(C,f)$ the image of the induced map $\Homc(C,f)\colon\Homc(C,X)\to\Homc(C,Y)$. Generalizing
\cite[Corollary~XI.3.4]{AuslanderReitenSmalo1995Representation}
to an exact category $\CC$, we obtain the following existence theorem; compare
\cite[Theorem~XI.3.6]{AuslanderReitenSmalo1995Representation}
and
\cite[Theorem~4.4]{Ringel2013Auslander}.
Since $\CC$ does not have enough projectives, the treatment here is completely different.

\begin{theorem}\label{thm:exist}
  Let $C$ and $Y$ be two objects. Assume that $C\in\CC_r$ and $H$ is a left $\Gamma_C$-submodule of $\Homc(C,Y)$ satisfying $\CP(C,Y)\subseteq H$. Then there exists some deflation $\alpha\colon X \to Y$, which is right $C$-determined such that $\Ker\alpha\in\add(\tau C)$ and $H=\Im\Homc(C,\alpha)$.
\end{theorem}

\begin{proof}
  By the assumption we have $\tau C\in\CC_l$ and $\tau^-\tau C\simeq C$ in $\underline\CC$. Then there exists a natural isomorphism
  \[
    \phi\colon\Homcu(C,-)
    \longrightarrow D\Extc(-,\tau C).
  \]
  Set $\gamma=\phi_C(\underline{\Id_C})$. By the naturality of $\phi$, for each object $Z$ and each morphism $f\colon C\to Z$, we have
  \[
    \phi_Z(\underline f)
    = D\Extc(f,\tau C)(\gamma)
    = \gamma\circ\Extc(f,\tau C).
  \]
  Then for each $\mu\in\Extc(Z,\tau C)$, we have
  \[
    \phi_Z(\underline f)(\mu)
    = \gamma(\Extc(f,\tau C)(\mu))
    = \gamma(\mu.f).
  \]

  We set $\underline H=H/\CP(C,Y)$ and set
  \[
    \underline H^\perp
    = \set{\mu\in\Extc(Y,\tau C) \middle|
    \phi_Y(\underline h)(\mu) = 0
    \mbox{ for each }
    \underline h \in \underline H}.
  \]
  We observe that $\underline H^\perp$ is a right $\Gamma_C$-submodule of $\Extc(Y,\tau C)$. Here, for any $f\in\Gamma_C$ and $\mu\in\Extc(Y,\tau C)$, the action of $f$ on $\mu$ is given by $\tau(\underline{f}).\mu$.
  Then there exists finitely many $\eta_1,\eta_2,\dots,\eta_n$ in $\Extc(Y,\tau C)$ such that $\underline H^\perp = \sum_{i=1}^n \eta_i\Gamma_C$. Assume that $\eta_i$ is represented by the conflation $\tau C\to X_i\xrightarrow{\alpha_i} Y$ for each $i=1,2,\dots,n$. We have that $\alpha_i$ is right $\tau^-\tau C$-determined by Lemma~\ref{lem:ker.det}. Then $\alpha_i$ is right $C$-determined by Proposition~\ref{prop:det}. We observe that $\bigoplus_{i=1}^n \alpha_i$ is a deflation. By Lemma~\ref{lem:det.sum}, we have that $\bigoplus_{i=1}^n \alpha_i$ is right $C$-determined.

  Consider the following commutative diagram obtained by a pullback
  \[\xymatrix@+1pc{
    \bigoplus_{i=1}^n \tau C\ar@{-->}[r] &X\ar@{-->}[r]^\alpha\ar@{-->}[d]  &Y\\
    \bigoplus_{i=1}^n \tau C\ar[r] &\bigoplus_{i=1}^n X_i &\bigoplus_{i=1}^n Y.
    \ar[u];[]^-{\smatrix{\Id_Y\\[-3pt]\vdots\\[3pt] \Id_Y}}
    \ar[l];[]^{\bigoplus_{i=1}^n \alpha_i}
    \ar@{=}[llu];[ll]
  }\]
  We have that $\alpha$ is a deflation and $\Ker\alpha$ lies in $\add(\tau C)$. By Lemma~\ref{lem:det.PB}, the deflation $\alpha$ is right $C$-determined. By a direct verification, we have
  \[
    \Im \Homc(C,\alpha)
    = \bigcap_{i=1}^n \Im\Homc(C,\alpha_i).
  \]

  For each $i=1,2,\dots,n$, we set
  \[
    ^\perp(\eta_i \Gamma_C)
    =\set{\underline h \in \Homcu(C,Y) \middle|
    \phi_Y(\underline{h})(\mu) = 0
    \mbox{ for each }
    \mu \in \eta_i \Gamma_C}.
  \]
  We observe that $^\perp(\eta_i\Gamma_C)$ is a left $\Gamma_C$-submodule of $\Homcu(C,Y)$. We mention that $\CP(C,Y)\subseteq\Im\Homc(C,\alpha_i)$, since $\alpha_i$ is a deflation.
  We claim that
  \[
    ^\perp(\eta_i\Gamma_C)
    = \Im \Homc(C,\alpha_i)/\CP(C,Y).
  \]

  Let $h\colon C\to Y$ be a morphism in $\Im\Homc(C,\alpha_i)$. We obtain that $\eta_i.h$ splits. Then we have
  \[
    \phi_Y(\underline h)(\tau(\underline f).\eta_i)
    = \gamma((\tau(\underline f).\eta_i).h)
    = \gamma(\tau(\underline f).(\eta_i.h))
    = 0,
  \]
  for each $f\colon C\to C$. It follows that
  $\Im\Homc(C,\alpha_i)/\CP(C,Y)\subseteq {^\perp(\eta_i\Gamma_C)}$.

  On the other hand, let $h\colon C\to Y$ be a morphism such that $\underline h\in{^\perp(\eta_i\Gamma_C)}$. Then we have $\phi_Y(\underline h)(\tau(\underline f).\eta_i)=0$ for each $f\colon C\to C$. By the definition of $\tau$ in \cite[Section~3]{JiaoGeneralized}, we have the following commutative diagram
  \[\xymatrix{
    \Homcu(C,Y)\ar[r]^-{\phi_Y} &D\Extc(Y,\tau C)\\
    \Homcu(C,Y)\ar[r]^-{\phi_Y} &D\Extc(Y,\tau C).
    \ar[lu];[l]_{\Homcu(\underline f,Y)}
    \ar[u];[]^{D\Extc(Y,\tau(\underline f))}
  }\]
  By a diagram chasing, we have
  \[\begin{split}
    \phi_Y(\underline{h} \circ \underline{f})
    &= \phi_Y ( \Homcu(\underline{f}, Y) (\underline{h}) )\\
    &= (D\Extc(Y, \tau(\underline{f})) \circ \phi_Y) (\underline{h})\\
    &= \phi_Y(\underline{h}) \circ \Extc(Y, \tau(\underline{f})).\\
  \end{split}\]
  Then for each $\eta_i$, we have
  \[
    \phi_Y(\underline h\circ\underline f)(\eta_i)
    = \phi_Y(\underline h)(\tau(\underline f).\eta_i)
    = 0.
  \]
  It follows that
  \[
    \phi_C(\underline f)(\eta_i.h)
    = \gamma((\eta_i.h).f)
    = \gamma(\eta_i.(h\circ f))
    = \phi_Y(\underline{h\circ f})(\eta_i)
    = 0.
  \]
  We observe that $\phi_C(\underline f)$ runs over all maps in $D\Extc(C,\tau C)$, when $\underline f$ runs over all morphisms in $\underline\End_\CC(C)$. It follows that $\eta_i.h$ splits. In other words, the morphism $h$ factors through $\alpha_i$. We then obtain $h\in \Im\Homc(C,\alpha_i)$. It follows that
  ${^\perp(\eta_i\Gamma_C)}\subseteq\Im\Homc(C,\alpha_i)/\CP(C,Y)$.

  We observe that
  \[
    \underline H
    = {^\perp(\underline H^\perp)}
    = {^\perp(\sum_{i=1}^n\eta_i\Gamma_C)}
    = \bigcap_{i=1}^n {^\perp(\eta_i\Gamma_C)},
  \]
  where the first equality follows from the isomorphism $\phi_Y$. It follows that
  \[
    \underline H
    = \bigcap_{i=1}^n \Im \Homc(C,\alpha_i)/\CP(C,Y)
    = \Im \Homc(C,\alpha)/\CP(C,Y).
  \]
  Then the result follows since $\CP(C,Y)\subseteq H$.
\end{proof}

\section{A characterization for determined deflation}

Let $\CC$ be a Hom-finite Krull-Schmidt exact category. We give a characterization for a deflation being right $C$-determined for some object $C$.

The following lemma, which is due to
\cite[Proposition~XI.2.4 and Lemma~XI.2.1]{AuslanderReitenSmalo1995Representation},
shows that each right $C$-determined morphism has a minimal right determiner.

\begin{lemma}\label{lem:min.det}
  Let $f\colon X\to Y$ be a morphism, which is right $C$-determined for some object $C$.
  \begin{enumerate}
    \item\label{item:lem:C:1}
      Assume that $C'$ is an indecomposable object and $f'\colon C'\to Y$ almost factors through $f$. Then $C'$ is a direct summand of $C$.
    \item\label{item:lem:C:2}
      Assume that $\set{C_1,C_2,\dots,C_t}$ is a complete set of pairwise non-isomorphic indecomposable objects such that there exists some morphism $f_i\colon C_i\to Y$, which almost factors through $f$. Then $\bigoplus_{i=1}^t C_i$ is a minimal right determiner of $f$.
    \qed
  \end{enumerate}
\end{lemma}

Recall that a morphism $f\colon X\to Y$ is called \emph{right minimal}, if each $g\in\Endc(X)$ with $f\circ g=f$ is an automorphism. Dually, we call $f$ \emph{left minimal}, if each $g\in\Endc(Y)$ with $g\circ f=f$ is an automorphism.

Two morphisms $f\colon X\to Y$ and $f'\colon X'\to Y$ are called \emph{right equivalent} if $f$ factors through $f'$ and $f'$ factors through $f$. Assume that $f$ and $f'$ are right equivalent. Given an object $C$, we mention that $f$ is right $C$-determined if and only if so is $f'$. We observe that $\CC$ has split idempotents;
see \cite[Corollary~4.4]{Krause2015Krull}.
By \cite[Proposition~7.6]{Buhler2010Exact},
we have that $f$ is a deflation if and only if so is $f'$.

Given a morphism $f\colon X\to Y$ , we call a right minimal morphism $f'\colon X'\to Y$ the \emph{right minimal version} of $f$, if $f$ and $f'$ are right equivalent. Following
\cite[Section~2]{Ringel2012Morphisms},
we call $\Ker f'$ the \emph{intrinsic kernel} of $f$.

We mention that each morphism $f\colon X\to Y$ has a right minimal version; see
\cite[Theorem~1]{Bian2009Right}.
Indeed, by the projectivization in
\cite[Section~II.2]{AuslanderReitenSmalo1995Representation},
we have an equivalence $F \colon \add(X \oplus Y) \to \proj \Gamma_{X \oplus Y}$, where $\proj \Gamma_{X \oplus Y}$ is the category of finitely generated projective $\Gamma_{X \oplus Y}$-modules. By
\cite[Proposition~I.2.1]{AuslanderReitenSmalo1995Representation},
the morphism $F(f)$ has a right minimal version $g$ in $\proj \Gamma_{X \oplus Y}$. Then $F^{-1}(g)$ yields the required morphism, where $F^{-1}$ is a quasi-inverse of $F$.

Dually, two morphisms $f\colon X\to Y$ and $f'\colon X\to Y'$ are left equivalent if $f$ factors through $f'$ and $f'$ factors through $f$. We mention that there exists some left minimal morphism $g\colon X\to Z$ such that $f$ and $g$ are left equivalent. We call $g$ the \emph{left minimal version} of $f$, and call $\Cok g$ the \emph{intrinsic cokernel} of $f$.

\begin{lemma}\label{lem:exist.ass}
  Let $\alpha\colon X\to Y$ be a deflation and $C$ be an indecomposable object. Assume that there exists some morphism $f\colon C\to Y$ which almost factors through $\alpha$. Then there exists an almost split conflation $K\to E\to C$ such that $K$ is a direct summand of the intrinsic kernel of $\alpha$.
\end{lemma}

\begin{proof}
  We may assume that $\alpha$ is right minimal. Let $\beta\colon Z\to C$ be the pullback of $\alpha$ along $f$. Then $\beta$ is a deflation. By Proposition~\ref{prop:ras}, we have that $\beta$ is right almost split. Let $\gamma\colon E\to C$ be the right minimal version of $\beta$. We observe that $\gamma$ is a right almost split deflation and $\Ker\gamma$ is a direct summand of $\Ker\beta\simeq\Ker\alpha$.

  We claim that $\delta \colon \Ker\gamma \xrightarrow{\phi} E\xrightarrow{\gamma} C$ is an almost split conflation;
  see \cite[Proposition~II.4.4]{Auslander1978Functors}.
  Indeed, let $f \colon \Ker\gamma \to K'$ be a non-section. Assume that $f$ does not factor through $\phi$. Then the conflation $f.\delta$ does not split. In particular, the deflation $\gamma'$ does not split, and hence factors through $\gamma$. We then obtain the following commutative diagram
  \[\xymatrix{
    \llap{$\delta \colon$ } \Ker\gamma \ar[r]^\phi \ar[d]_f
    &E \ar[r]^\gamma \ar[d]_g
    &C \ar@{=}[d]
    \\
    \llap{$f.\delta \colon$\quad} K' \ar[r]^{\phi'} \ar@{-->}[d]_{f'}
    &E' \ar[r]^{\gamma'} \ar@{-->}[d]_{g'}
    &C \ar@{=}[d]
    \\
    \llap{$\delta \colon$ } \Ker\gamma \ar[r]^\phi
    &E \ar[r] ^\gamma
    &C.
  }\]
  Since $\gamma$ is right minimal, the morphism $g' \circ g$ is an isomorphism. We obtain that $f' \circ f$ is an isomorphism, which is a contradiction since $f$ is not a section. It follows that $f$ factors through $\phi$. We obtain that $\phi$ is left almost split, and then $\delta$ is an almost split conflation.
\end{proof}

\begin{corollary}\label{cor:min.det}
  A minimal right determiner of a deflation has no nonzero projective direct summands and lies in $\CC_r$.
\end{corollary}

\begin{proof}
  Let $C$ be a minimal right determiner of a deflation $\alpha\colon X\to Y$. It is sufficient to show that each indecomposable direct summand $C'$ of $C$ is non-projective and lies in $\CC_r$. By Lemma~\ref{lem:min.det}, there exists some morphism $f\colon C'\to Y$ which almost factors through $\alpha$. We observe that $C'$ is not projective, since $f$ does not factor through $\alpha$. By Lemma~\ref{lem:exist.ass}, there exists an almost split conflation ending at $C'$. Then the result follows from Lemma~\ref{lem:ass}(\ref{item:lem:X:1}).
\end{proof}

It is well known that in the category of finitely generated modules over an artin algebra, each morphism is right $C$-determined for some object $C$; see \cite[Corollary~XI.1.4]{AuslanderReitenSmalo1995Representation} and \cite[Theorem~1]{Ringel2012Morphisms}. Here, we can only give a characterization for a deflation being right $C$-determined for some object $C$.

\begin{theorem}\label{thm:det}
  A deflation $\alpha\colon X\to Y$ is right $C$-determined for some object $C$ if and only if the intrinsic kernel of $\alpha$ lies in $\CC_l$.
\end{theorem}

\begin{proof}
  For the sufficiency, let $\alpha'\colon X'\to Y$ be the right minimal version of $\alpha$. We observe that $\alpha'$ is still a deflation and $\Ker\alpha'\in\CC_l$. By Lemma~\ref{lem:ker.det}, we have that $\alpha'$ is right $\tau^-(\Ker\alpha')$-determined. It follows that $\alpha$ is right $\tau^-(\Ker\alpha')$-determined since $\alpha$ and $\alpha'$ are right equivalent.

  For the necessity, we may assume that $C$ is a minimal right determiner of $\alpha$. By Corollary~\ref{cor:min.det}, we have $C\in\CC_r$. We observe that $\Im\Homc(C,\alpha)$ is a $\Gamma_C$-submodule of $\Homc(C,Y)$. Since $\alpha$ is a deflation, we have $\CP(C,Y)\subseteq\Im\Homc(C,\alpha)$. By Theorem~\ref{thm:exist}, there exists some deflation $\alpha'\colon X'\to Y$, which is right $C$-determined such that $\Ker\alpha'\in\add(\tau C)$ and $\Im\Homc(C,\alpha)=\Im\Homc(C,\alpha')$. We observe that $\alpha\circ f$ factors through $\alpha'$ and $\alpha'\circ f'$ factors through $\alpha$ for any $f\colon C\to X$ and $f'\colon C\to X'$. Since $\alpha$ and $\alpha'$ are right $C$-determined, we have that $\alpha$ and $\alpha'$ factor through each other. It follows that $\alpha$ and $\alpha'$ are right equivalent. The intrinsic kernel of $\alpha$ is also the intrinsic kernel of $\alpha'$, which is a direct summand of $\Ker\alpha'$. Then the result follows since $\Ker \alpha' \in \add (\tau C)$ and $\tau C \in \CC_l$.
\end{proof}

\begin{example}\label{ex:det}
  Let $Q$ be the following infinite quiver
  \[
    \underset1\circ \longrightarrow
    \underset2\circ \longrightarrow
    \cdots \longrightarrow
    \underset n\circ \longrightarrow\cdots.
  \]
  We consider the representations of $Q$ over a field. For each $i\geq1$, we denote by $P_i$ the indecomposable projective representation and by $S_i$ the simple representation corresponding to $i$. Let $\CC$ be the category of finitely presented representations. It is well known that $\CC$ is a Hom-finite Krull-Schmidt abelian category.

  For each $i\geq1$, we consider the projective cover $f_i\colon P_i\to S_i$. We have $\Ker f_i\simeq P_{i+1}$. By \cite[Proposition~4.4]{JiaoGeneralized}, we have that $P_{i+1}\not\in\CC_l$. By Theorem~\ref{thm:det}, we have that $f_i$ is not right $C$-determined for any object $C$.

  By Lemma~\ref{lem:exist.ass}, for any indecomposable object $C$, any morphism $g\colon C\to S_i$ does not almost factor through $f_i$. We observe that $f_i$ is not right 0-determined. Then we have that the condition ``$f$ is right $C$-determined for some object $C$'' for Lemma~\ref{lem:min.det}(\ref{item:lem:C:2}) is necessary.
\end{example}

\section{More descriptions of objects in $\CC_r$ and $\CC_l$}

Let $\CC$ be a Hom-finite Krull-Schmidt exact category. We will give some characterizations for an object lying in $\CC_r$ or $\CC_l$ via morphisms determined by objects.

\begin{proposition}\label{prop:C_l}
  Let $K$ be an object without nonzero injective direct summands. Then $K$ lies in $\CC_l$ if and only if there exists some deflation $\alpha\colon X\to Y$, which is right $C$-determined for some object $C$ such that $K$ is the intrinsic kernel of $\alpha$.
\end{proposition}

\begin{proof}
  The sufficiency follows from Theorem~\ref{thm:det}. For the necessity, we assume $K\in\CC_l$. Decompose $K$ as the direct sum of indecomposable objects $K_1,K_2,\dots,K_n$. We have that each $K_i$ is non-injective. Then for each $i=1,2,\dots,n$, there exists a non-split conflation $K_i\to X_i \xrightarrow{\alpha_i} Y_i$. We observe that $\bigoplus_{i=1}^n \alpha_i$ is a deflation. By Lemma~\ref{lem:ker.det}, we have that $\bigoplus_{i=1}^n \alpha_i$ is right $\tau^-K$-determined. We observe that $\alpha_i$ is right minimal. Then $\bigoplus_{i=1}^n \alpha_i$ is also right minimal. It follows that $K$ is the intrinsic kernel of $\bigoplus_{i=1}^n \alpha_i$.
\end{proof}

The following lemma is the converse of Theorem~\ref{thm:exist}.

\begin{lemma}\label{lem:C_r}
  Let $C$ be an object. Assume that for each object $Y$ and each $\Gamma_C$-submodule $H$ of $\Homc(C,Y)$ satisfying $\CP(C,Y)\subseteq H$, there exists some deflation $\alpha\colon X\to Y$, which is right $C$-determined such that $H=\Im\Homc(C,\alpha)$. Then the object $C$ lies in $\CC_r$.
\end{lemma}

\begin{proof}
  It is sufficient to show that each non-projective indecomposable direct summand $C'$ of $C$ lies in $\CC_r$. Then each $f\in\CP(C,C')$ is not a retraction. We observe that $\radc(C,C')$ is formed by non-retractions. We obtain $\CP(C,C')\subseteq\radc(C,C')$. We observe that $\radc(C,C')$ is a $\Gamma_C$-submodule of $\Homc(C,C')$. By the assumption, there exists some deflation $\alpha\colon X\to C'$, which is right $C$-determined such that $\radc(C,C')=\Im\Homc(C,\alpha)$.

  We claim that $\Id_{C'}$ almost factors through $\alpha$. Indeed, the deflation $\alpha$ is not a retraction, since $\Im\Homc(C,\alpha) = \radc(C,C')$ is a proper submodule of $\Homc(C,C')$. It follows that $\Id_{C'}$ does not factor through $\alpha$. Let $f\colon T\to C'$ be a morphism in $\radc(T,C')$. For each $g\colon C\to T$, the morphism $f\circ g$ lies in $\radc(C,C')=\Im\Homc(C,\alpha)$. It follows that $f\circ g$ factors through $\alpha$. Since $\alpha$ is right $C$-determined, we have that $f$ factors through $\alpha$. It follows that $\Id_{C'}$ almost factors through $\alpha$. By Lemma~\ref{lem:exist.ass}, there exists an almost split conflation ending at $C'$. Then the result follows from Lemma~\ref{lem:ass}(\ref{item:lem:X:1}).
\end{proof}

Collecting the results obtained so far, we list some some characterizations for an object lying in $\CC_r$. We mention that the equivalence between Theorem~\ref{thm:C}(\ref{item:thm:C:5}) and Theorem~\ref{thm:C}(\ref{item:thm:C:6}) is somehow surprising.

\begin{theorem}\label{thm:C}
  Let $\CC$ be a Hom-finite Krull-Schmidt exact category and let $C\in\CC$. The following statements are equivalent.
  \begin{enumerate}
    \item\label{item:thm:C:1}
      The object $C$ lies in $\CC_r$.
    \item\label{item:thm:C:2}
      For each object $Y$ and each $\Gamma_C$-submodule $H$ of $\Homc(C,Y)$ satisfying $\CP(C,Y)\subseteq H$, there exists some deflation $\alpha\colon X\to Y$, which is right $C$-determined such that $H=\Im\Homc(C,\alpha)$.
    \item\label{item:thm:C:3}
      Each inflation $\alpha\colon X\to Y$ with $\Cok\alpha\in\add C$ is left $K$-determined for some object $K$.
  \end{enumerate}
  If moreover $C$ is non-projective indecomposable, they are equivalent to the following statements.
  \begin{enumerate}
    \setcounter{enumi}{3}
    \item\label{item:thm:C:4}
      There exists an inflation $\alpha\colon X\to Y$ whose intrinsic cokernel is $C$ such that $\alpha$ is left $K$-determined for some object $K$.
    \item\label{item:thm:C:5}
      There exists an almost split conflation ending at $C$.
    \item\label{item:thm:C:6}
      There exists a non-split deflation which is right $C$-determined.
    \item\label{item:thm:C:7}
      There exists a deflation $\alpha\colon X\to Y$ and a morphism $f\colon C\to Y$ such that $f$ almost factors through $\alpha$.
\end{enumerate}
\end{theorem}

\begin{proof}
  By Theorem~\ref{thm:exist} and Lemma~\ref{lem:C_r}, we have
  ``(\ref{item:thm:C:1}) $\Leftrightarrow$ (\ref{item:thm:C:2})''.

  By the dual of Lemma~\ref{lem:ker.det}, we have
  ``(\ref{item:thm:C:1}) $\Rightarrow$ (\ref{item:thm:C:3})''.
  For each indecomposable non-projective direct summand $C'$ of $C$, there exists some non-split inflation $\alpha \colon X \to Y$ such that $\Cok \alpha \simeq C'$. We observe that $\alpha$ is left minimal, since $\alpha$ is non-split and $C'$ is indecomposable. By the dual of Theorem~\ref{thm:det}, we have that $C' \in \CC_r$. Then
  ``(\ref{item:thm:C:3}) $\Rightarrow$ (\ref{item:thm:C:1})''
  follows.

  Now, we assume that $C$ is non-projective indecomposable.
  Then the dual of Proposition~\ref{prop:C_l} implies
  ``(\ref{item:thm:C:1}) $\Leftrightarrow$ (\ref{item:thm:C:4})'',
  and Lemma~\ref{lem:ass}(\ref{item:lem:X:1}) implies
  ``(\ref{item:thm:C:1}) $\Leftrightarrow$ (\ref{item:thm:C:5})''.

  It is well known that the right almost split deflation ending at $C$ is non-split and right $C$-determined; also see Lemma~\ref{lem:ker.det} and Proposition~\ref{prop:det}. Then we have
  ``(\ref{item:thm:C:5}) $\Rightarrow$ (\ref{item:thm:C:6})''.
  Let $\alpha$ be a non-split right $C$-determined deflation. We observe by Lemma~\ref{lem:min.det} that $C$ is a minimal right determiner of $\alpha$. By Lemma~\ref{lem:min.det}, we have
  ``(\ref{item:thm:C:6}) $\Rightarrow$ (\ref{item:thm:C:7})''.
  By Lemma~\ref{lem:exist.ass}, we have
  ``(\ref{item:thm:C:7}) $\Rightarrow$ (\ref{item:thm:C:5})''.
\end{proof}

\section{Exact categories having Auslander-Reiten duality}

Let $\CC$ be a Hom-finite Krull-Schmidt exact category. Following \cite[Definition~2.6]{Krause2013Morphisms} and \cite[Definition~3.1]{ChenLe2015note}, we introduce the following notion.

\begin{definition}\label{def:det}
  An exact category $\CC$ is said to \emph{have right stably determined deflations} if for each object $Y$ the following conditions hold.
  \begin{enumerate}
    \item\label{item:def:1}
      Each deflation ending at $Y$ is right $C$-determined for some object $C$.
    \item\label{item:def:2}
      For each object $C$ and each $\Gamma_C$-submodule $H$ of $\Homc(C,Y)$ such that $\CP(C,Y)\subseteq H$, there exists some deflation $\alpha\colon X\to Y$, which is right $C$-determined such that $H=\Im\Homc(C,\alpha)$.
    \qed
  \end{enumerate}
\end{definition}

Dually, an exact category $\CC$ is said to \emph{have left stably determined inflations} if the opposite category $\CC^{\op}$ has right stably determined deflations.

We give the following characterizations for an exact category having Auslander-Reiten duality in the sense of \cite{LenzingZuazua2004Auslander}.

\begin{theorem}\label{thm:AR}
  Let $\CC$ be a Hom-finite Krull-Schmidt exact category. The following statements are equivalent.
  \begin{enumerate}
    \item\label{item:thm:AR:1}
      $\CC$ has Auslander-Reiten duality.
    \item\label{item:thm:AR:2}
      $\CC$ has right stably determined deflations.
    \item\label{item:thm:AR:3}
      $\CC$ has left stably determined inflations.
  \end{enumerate}
\end{theorem}

\begin{proof}
  We observe that $\CC$ has Auslander-Reiten duality in the sense of \cite{LenzingZuazua2004Auslander} if and only if $\CC_l=\CC=\CC_r$.
  The dual of the equivalence between Theorem~\ref{thm:C}(\ref{item:thm:C:1}) and Theorem~\ref{thm:C}(\ref{item:thm:C:3}) implies that an object $K$ lies in $\CC_l$ if and only if each deflation whose kernel lies in $\add K$ is right $C$-determined for some $C$. Combining this with the equivalence between Theorem~\ref{thm:C}(\ref{item:thm:C:1}) and Theorem~\ref{thm:C}(\ref{item:thm:C:2}), we have
  ``(\ref{item:thm:AR:1}) $\Leftrightarrow$ (\ref{item:thm:AR:2})''.

  By duality, we have
  ``(\ref{item:thm:AR:1}) $\Leftrightarrow$ (\ref{item:thm:AR:3})''.
\end{proof}

\begin{remark}
  Compared with \cite[Theorem~3.4]{ChenLe2015note}, it is somewhat surprising that Theorem~\ref{thm:AR}(\ref{item:thm:AR:2}) and Theorem~\ref{thm:AR}(\ref{item:thm:AR:3}) are equivalent. It seems that Definition~\ref{def:det}(\ref{item:def:1}) and Definition~\ref{def:det}(\ref{item:def:2}) are more symmetric than the conditions (REC1) and (REC2) in \cite[Definition~3.1]{ChenLe2015note}.
\end{remark}

\section*{Acknowledgements}

We are very grateful to the referee for many helpful suggestions.
The first author thanks his advisor Professor~Xiao-Wu Chen for his guidance and encouragement.
We thank Dawei Shen for helpful comments.
The work is supported by National Natural Science Foundation of China (Nos. 11571329 and 11671174).


\end{document}